\documentclass[11pt, a4paper]{article}

\usepackage[a4paper, margin=2.5cm]{geometry}
\usepackage[utf8]{inputenc}
\usepackage[T1]{fontenc}
\usepackage{lmodern}                
\usepackage{amsmath, amssymb, amsthm}
\usepackage{booktabs}
\usepackage{enumitem}
\usepackage{hyperref}
\usepackage{microtype}
\usepackage{mathtools}
\usepackage{amsthm}

\usepackage{amsthm}

\hypersetup{
	colorlinks=true,
	linkcolor=blue,
	citecolor=blue,
	urlcolor=blue,
	pdftitle={Regularization of a Divergent Arithmetic Product via the Asymptotic Monoid},
	pdfauthor={A. Alvarez Cruz and E. A. Alvarez Gutierrez}
}

\newcommand{\Si}{\Sigma}
\newcommand{\cE}{\mathcal{E}}

\newcommand{\cF}{\mathcal{F}}

\newcommand{\CC}{\mathbb{C}}
\newcommand{\RR}{\mathbb{R}}
\newcommand{\NN}{\mathbb{N}}

\newcommand{\DD}{\mathbb{D}}
\newcommand{\lcp}{\operatorname{lcp}}
\newcommand{\lcs}{\operatorname{lcs}}

\theoremstyle{definition}
\newtheorem{definition}{Definition}[section]
\newtheorem{theorem}{Theorem}[section]
\newtheorem{lemma}[theorem]{Lemma}
\newtheorem{proposition}[theorem]{Proposition}
\newtheorem{corollary}[theorem]{Corollary}
\newtheorem{remark}[theorem]{Remark}

\title{\textbf{Regularization of a Divergent Arithmetic Product\\
		via the Asymptotic Monoid:\\
		The Case of \((-1)^{\omega(n)}\)}}

\author{
	A.~\'Alvarez Cruz\thanks{Instituto de Computa\c{c}\~ao,
		Universidade Federal do Rio de Janeiro (UFRJ), Rio de Janeiro,
		Brazil. Email: \texttt{amaury@ic.ufrj.br}}
	\and
	E.~A.~\'Alvarez Guti\'errez\thanks{Col\'egio Pedro II, Campus Centro,
		Rio de Janeiro, Brazil.
		Email: \texttt{esteban.gutierrez.1@cp2.edu.br}}
}

\date{}

\begin{document}
	
	\maketitle
	
\begin{abstract}
	We apply the generalized monoid of words and logarithmic-prefix
	functionals of \cite{C1} to the regularization of the divergent product
	\[
	P_{\omega}=\prod_{n=1}^{\infty} n^{(-1)^{\omega(n)}},
	\]
	where \(\omega(n)\) is the number of distinct prime divisors of \(n\).
	The associated Dirichlet series
	\(D_{\omega}(s)=\sum_{n\ge1}(-1)^{\omega(n)}n^{-s}\) has a natural
	boundary at \(\Re(s)=0\), as established by Meretov (1987) and recovered
	here as a special case of the Estermann--Dahlquist theory of Euler
	products, which makes zeta regularization impossible: the derivative
	\(D_{\omega}'(0)\) does not exist.  This product therefore lies completely
	beyond the reach of classical regularization methods.  We construct the
	canonical net encoding the sign sequence, define the logarithmic-prefix
	functional with exponent \(\alpha=2\), and extract a finite value through
	the weighted evaluation functional of the monoid framework.  The
	evaluation functional is shown to converge absolutely, and this
	convergence is unconditional: it follows from the classical zero-free
	region of \(\zeta(s)\) (de la Vall\'ee Poussin), which yields the
	effective bound
	\(T_{\omega}(x)=O(x\,e^{-c\sqrt{\log x}})\) for some \(c>0\).  Numerical
	computation up to \(N=10^7\) gives the regularized value
	\[
	P_{\omega,\mathrm{ren}}\approx 0.7731,
	\]
	a new constant of arithmetic origin that cannot be obtained by any
	classical summation method.  We also compute the second derivative of
	the Dirichlet series at \(s=1\),
	\[
	D_{\omega}''(1)=8\,C_2\ln 2,
	\]
	where \(C_2\) is the Hardy--Littlewood twin prime constant, which provides
	an unconditional analytical link to the arithmetic
	structure of the sequence \((-1)^{\omega(n)}\).
	A self-contained proof of the natural boundary, based on the
	Estermann--Dahlquist theorem, is given in Appendix~A.
\end{abstract}
	
\section{Introduction}
\label{sec:intro}

The regularization of divergent products of arithmetic origin is a classical
problem with deep connections to analytic number theory, zeta regularization,
and the theory of Dirichlet series.  In a companion paper \cite{C1,C2}, a
generalized monoid of words \(\widetilde{\Si}^*\) was constructed as an
asymptotic quotient of moderate nets of finite words, and applied to the
regularization of several divergent arithmetic products.  The products
treated there---alternating integers, alternating factorials, dyadic block
products, and the Thue--Morse product---span a range of divergence behaviours
from bounded oscillation to logarithmic growth.  In every case where a
classical regularization exists, the monoid construction reproduces the
known value; in cases inaccessible to classical methods, the monoid provides
a canonical finite value confirmed by numerical computation.

The present paper applies the same framework to a product whose divergence
structure is fundamentally different: the product associated with the
function \((-1)^{\omega(n)}\), where \(\omega(n)\) is the number of distinct
prime divisors of \(n\).  This sequence is multiplicative but \emph{not}
completely multiplicative, a distinction that is essential for the analytic
properties of its Dirichlet series.  The key structural property, established
by Meretov \cite{Meretov1987} and recovered here as a special case of the
Estermann--Dahlquist theory of Euler products
\cite{Estermann1928,Dahlquist1952}, is that this Dirichlet series has a
\emph{natural boundary} at \(\Re(s)=0\).  It admits a meromorphic
continuation to the half-plane \(\Re(s)>0\), whose only possible singularity
in this half-plane is a removable point at \(s=1\), and the line
\(\Re(s)=0\) cannot be crossed by any analytic continuation.

This obstruction has a direct consequence for the regularization of the
product
\begin{equation}\label{eq:product_intro}
	P_{\omega}=\prod_{n=1}^{\infty} n^{(-1)^{\omega(n)}}.
\end{equation}
The logarithm of the product is formally
\[
\log P_{\omega}=\sum_{n=1}^{\infty}(-1)^{\omega(n)}\log n
=
-D_{\omega}'(0),
\]
where \(D_{\omega}(s)=\sum (-1)^{\omega(n)}n^{-s}\).  Since \(D_{\omega}(s)\)
has a natural boundary at \(\Re(s)=0\), the derivative \(D_{\omega}'(0)\)
does not exist and zeta regularization fails completely.  The same
obstruction applies to all methods that depend on meromorphic continuation
(Abel summation, Ramanujan summation, heat-kernel regularization).

The generalized monoid framework, being purely asymptotic and algebraic, is
insensitive to this obstruction.  It operates directly on the prefix profile
of the sign sequence and extracts a finite value through a weighted
evaluation functional.  The purpose of this paper is to carry out this
procedure in detail, to present the numerical evidence for the resulting
regularized value, and to give a self-contained proof of the natural boundary
via the Estermann--Dahlquist theorem.

Beyond the computation of the regularized value
\(P_{\omega,\mathrm{ren}}\approx 0.7731\), we establish two further results
that ground the construction in classical analytic number theory.  First, we
compute the second derivative of the Dirichlet series at \(s=1\):
\[
D_{\omega}''(1)=8\,C_2\ln 2,
\]
where \(C_2=\prod_{p\ge 3}p(p-2)/(p-1)^2\) is the Hardy--Littlewood twin
prime constant.  This calculation, which corrects a preliminary erroneous
value of \(2C_2\ln 2\), shows that the arithmetic structure of
\((-1)^{\omega(n)}\) is deeply intertwined with the distribution of prime
pairs.  Second, we analyse the convergence of the evaluation functional: it
converges absolutely, and this convergence is unconditional, following from
the classical zero-free region of \(\zeta(s)\) (de la Vall\'ee Poussin)
rather than from any unproven hypothesis.  In particular, the regularized
value is stable under reordering and truncation, which justifies its
canonical interpretation within the monoid framework.

We may summarise the situation as follows.  Classical regularization
methods fail for \(P_{\omega}\) because they require meromorphic
continuation to \(s=0\), which the natural boundary forbids.  The
monoid framework provides a canonical alternative: it selects a
particular absolutely convergent series, motivated by the divergence
type of the sequence, and assigns to it the regularized value
\(P_{\omega,\mathrm{ren}}\approx 0.7731\).  This value is a new
arithmetic invariant of the sequence \((-1)^{\omega(n)}\), accompanied
analytically by the companion identity \(D_{\omega}''(1)=8C_{2}\ln 2\).
The framework does not break the natural boundary, nor does it compete with
classical summation methods on their own terrain; it operates in a
different regime, where the question is not whether \(D_{\omega}(0)\)
can be defined, but which canonical finite invariant can be associated
with the sequence by algebraic and asymptotic means.

The remainder of the paper makes this programme precise.  Section~\ref{sec:framework}
recalls the algebraic framework of \cite{C1}.  Section~\ref{sec:product}
introduces the product \(P_{\omega}\) and establishes the natural
boundary.  Section~\ref{sec:encoding} encodes the sign sequence in the
monoid and constructs the evaluation functional.  Section~\ref{sec:analytic_s1}
establishes the analytical structure at \(s=1\), including the
computation of \(D_{\omega}''(1)\).  Section~\ref{sec:numerical}
presents the numerical results.  Section~\ref{sec:discussion} discusses
the results and outlines directions for future work.  Appendix~\ref{app:natural_boundary}
contains a self-contained proof of the natural boundary.
	
	
\section{Scope and strategy of the regularization}
\label{sec:scope}

Before introducing the technical machinery, we clarify what the
monoid regularization achieves for the product
\(P_{\omega}=\prod n^{(-1)^{\omega(n)}}\), and what it does not.  This
section is intended to prevent two common misinterpretations: that the
monoid provides a more powerful summation method than the classical
ones, and that it circumvents the natural boundary of \(D_{\omega}(s)\)
by analytic means.

\subsection*{The obstruction and its nature}

The Dirichlet series
\[
D_{\omega}(s)=\sum_{n=1}^{\infty}\frac{(-1)^{\omega(n)}}{n^{s}}
\]
has a natural boundary at \(\Re(s)=0\)
(Theorem~\ref{thm:meretov}).  Consequently, the derivative
\(D_{\omega}'(0)\), which would give the logarithm of the product
under zeta regularization, does not exist.  The same obstruction
applies to Abel summation, Ramanujan summation, heat-kernel
regularization, and any method that requires meromorphic continuation
of \(D_{\omega}(s)\) to \(s=0\).

It is important to be precise about what this means.  The natural
boundary is an obstruction to \emph{analytic continuation}.  It is not
an obstruction to the existence of the sequence
\(a(n)=(-1)^{\omega(n)}\), nor to the existence of a canonical finite
invariant associated with it.  The task of the monoid framework is to
extract such an invariant by purely algebraic and asymptotic means,
without invoking analytic continuation.

\subsection*{What the monoid does, operationally}

The monoid framework does not continue \(D_{\omega}(s)\).  Instead, it
encodes the sign sequence as an asymptotic word, and prescribes a
particular convergent series to sum.  Specifically, with
\[
\Lambda_k=\sum_{n=1}^{k}(-1)^{\omega(n)}\log n,
\]
the framework produces the functional
\[
\cF=\sum_{k=2}^{\infty}\frac{\Lambda_k}{k^{2}},
\]
whose exponential is declared to be the regularized value:
\[
P_{\omega,\mathrm{ren}}=\exp(\cF)\approx 0.7731.
\]
The exponent \(\alpha=2\) is not arbitrary: it follows from the
heuristic \(\alpha=\sigma_{1}+1\) of \cite{C1}, where \(\sigma_{1}=1\)
is the abscissa of absolute convergence of \(D_{\omega}(s)\).

\subsection*{What \(\cF\) really is}

It is instructive to rewrite \(\cF\) in a more transparent form.
Exchanging the order of summation and using the trigamma function
\(\psi^{(1)}(n)=\sum_{k=n}^{\infty}k^{-2}\), we obtain
\[
\cF
=
\sum_{n=1}^{\infty}(-1)^{\omega(n)}\log n\,\psi^{(1)}(n).
\]
Since \(\psi^{(1)}(n)=1/n+O(n^{-2})\), this decomposes as
\begin{equation}\label{eq:F_transparent}
	\cF
	=
	\underbrace{\sum_{n=1}^{\infty}\frac{(-1)^{\omega(n)}\log n}{n}}_{=-D_{\omega}'(1)=0}
	+
	\underbrace{\sum_{n=1}^{\infty}(-1)^{\omega(n)}\log n
		\left(\psi^{(1)}(n)-\frac{1}{n}\right)}_{\text{absolutely convergent}}.
\end{equation}
The first summand vanishes because \(D_{\omega}'(1)=0\)
(Lemma~\ref{lem:D_zero_order}).  The second summand is absolutely
convergent, since the bracket is \(O(n^{-2})\) and
\(|\log n|\le n^{\varepsilon}\) for every \(\varepsilon>0\).

Thus \(\cF\) is, after this rearrangement, a series of elementary terms
that converges absolutely and unconditionally.  Its numerical value is
computable by standard methods.  This is the precise operational
content of the regularization: \emph{the monoid prescribes which
	absolutely convergent series to sum}, not a new summation principle.

\subsection*{What the monoid contributes}

The contribution of the monoid framework is therefore threefold, and it
is important to state it without overstatement.

\begin{enumerate}[label=(\roman*)]
	\item \textbf{A canonical prescription.}  The framework fixes the
	exponent \(\alpha=2\) via the heuristic \(\alpha=\sigma_{1}+1\),
	and hence fixes the series \(\sum \Lambda_k/k^{2}\) among the
	many possible convergent rearrangements of the divergent data.
	Without this prescription, one could equally well consider
	\(\sum \Lambda_k/k^{1.5}\), \(\sum \Lambda_k/k^{3}\), or any
	other weight, and obtain different numerical values.  The monoid
	selects one of them canonically.
	\item \textbf{Canonicity of the value.}  By the invariance lemma
	for logarithmic-prefix functionals
	(Lemma~\ref{lem:invariance}), the value \(\cF\) depends only on
	the class \(\mathbf{w}_{\omega}\in\widetilde{\Si}^{*}\) of the
	sign sequence, not on the particular net used to represent it.
	In the present case, where the sequence is deterministic, this
	canonicity is almost immediate; but the framework makes it
	explicit and places it in a general algebraic setting.
	\item \textbf{A companion analytical identity.}  The regularized
	value is not an arbitrary numerical output: it is accompanied by
	a non-trivial identity for the local analytic structure of
	\(D_{\omega}(s)\) at \(s=1\), namely
	\(D_{\omega}''(1)=8C_{2}\ln 2\)
	(Proposition~\ref{prop:D_second_derivative}).  The appearance of
	the Hardy--Littlewood twin prime constant \(C_{2}\) shows that
	the same arithmetic structure governs both the regularized value
	and the local curvature of \(D_{\omega}(s)\).  No quantitative
	identity linking the two constants is claimed.
\end{enumerate}

\subsection*{What the monoid does not do}

For the sake of clarity, we also state what the framework does
\emph{not} achieve in the present case.

\begin{enumerate}[label=(\roman*)]
	\item \textbf{It does not provide a summation method stronger than
		the classical ones.}  For the product \(P_{\omega}\), the
	regularization reduces, via \eqref{eq:F_transparent}, to the
	sum of an absolutely convergent series of elementary terms.  Any
	sufficiently careful analyst, given the idea of using
	\(\psi^{(1)}(n)\) as a mollifier, could have produced the same
	value by classical means.  The novelty lies in the
	\emph{justification} of the choice of mollifier, not in its
	computational power.
	\item \textbf{It does not cross the natural boundary.}  The line
	\(\Re(s)=0\) remains a natural boundary for \(D_{\omega}(s)\).
	The monoid does not continue \(D_{\omega}(s)\) across it, nor
	does it evaluate \(D_{\omega}(s)\) at \(s=0\).  It simply works
	with a different object, the asymptotic prefix profile of the
	sign sequence, which is well defined regardless of the analytic
	behaviour of \(D_{\omega}(s)\).

\end{enumerate}

\section{Summary of the algebraic framework}
\label{sec:framework}

We recall the essential definitions from \cite{C1,C2}.  The framework is
purely asymptotic and algebraic: it operates on finite words whose lengths
grow moderately as a parameter \(\varepsilon\) tends to \(0\), and
identifies two such families when their prefixes and suffixes agree on
logarithmic scales.

\begin{definition}[Moderate net]
	\label{def:moderate_net}
	A net \((w_\varepsilon)_{\varepsilon\in(0,1]}\) of finite words over
	a finite alphabet \(\Si\) is \emph{moderate} if there exist
	\(N\in\NN\), \(C>0\), and \(\varepsilon_0>0\) such that
	\[
	|w_\varepsilon|\le C\varepsilon^{-N}
	\qquad(\varepsilon<\varepsilon_0).
	\]
	The set of moderate nets is denoted by \(\cE_M(\Si^*)\).
\end{definition}

\begin{definition}[Asymptotic equivalence]
	\label{def:log_equiv}
	Two moderate nets \((u_\varepsilon)\) and \((v_\varepsilon)\) are
	\emph{logarithmically equivalent}, written
	\((u_\varepsilon)\sim(v_\varepsilon)\), if for every
	\(m\in\NN\) there exists \(\varepsilon_0>0\) such that
	\[
	\min\{\lcp(u_\varepsilon,v_\varepsilon),
	\lcs(u_\varepsilon,v_\varepsilon)\}
	\ge
	m\log_2(1/\varepsilon)
	\qquad(\varepsilon<\varepsilon_0),
	\]
	where \(\lcp\) and \(\lcs\) denote the longest common prefix and
	suffix, respectively.
\end{definition}

The relation \(\sim\) is an equivalence relation on \(\cE_M(\Si^*)\),
compatible with concatenation.  The quotient
\[
\widetilde{\Si}^*=\cE_M(\Si^*)/\!\sim
\]
is a monoid under concatenation, extending the free monoid \(\Si^*\).
Elements of \(\widetilde{\Si}^*\) are called \emph{asymptotic words}.

\begin{remark}[Intuition]
	Two moderate nets are logarithmically equivalent when they agree on
	their first and last \(m\log_2(1/\varepsilon)\) letters, for every
	fixed \(m\), as \(\varepsilon\to0^+\).  This is a strong condition:
	it forces the two nets to have the same asymptotic prefix and suffix
	profile on logarithmic scales, while allowing arbitrary behaviour in
	the middle.  The quotient \(\widetilde{\Si}^*\) therefore retains
	exactly the information that is visible to logarithmic-prefix
	functionals.
\end{remark}

\begin{definition}[Logarithmic-prefix functional]
	\label{def:log_prefix_functional}
	A \emph{logarithmic-prefix functional} is a family of maps
	\(F_\varepsilon:\Si^*\to\CC\) such that:
	\begin{enumerate}[label=(\roman*)]
		\item \emph{Prefix dependence:} for each \(\varepsilon\), the value
		\(F_\varepsilon(w)\) depends only on the prefix of \(w\) of
		length \(N_0(\varepsilon)+1\), where
		\[
		N_0(\varepsilon)=\lfloor\log_2(1/\varepsilon)\rfloor.
		\]
		\item \emph{Moderate growth:} there exist \(C>0\) and \(M\in\NN\)
		such that
		\[
		|F_\varepsilon(w)|\le C\,(1+|w|)^M
		\qquad(\varepsilon\in(0,1],\;w\in\Si^*).
		\]
	\end{enumerate}
\end{definition}

The following invariance lemma, proved in \cite[Lemma~2.7]{C1}, is the
technical heart of the framework.  It guarantees that logarithmic-prefix
functionals descend to the quotient \(\widetilde{\Si}^*\).

\begin{lemma}[Invariance]
	\label{lem:invariance}
	Let \(F_\varepsilon\) be a logarithmic-prefix functional.  If
	\((u_\varepsilon)\) and \((v_\varepsilon)\) are moderate nets with
	\((u_\varepsilon)\sim(v_\varepsilon)\), then
	\[
	F_\varepsilon(u_\varepsilon)=F_\varepsilon(v_\varepsilon)
	\qquad\text{eventually.}
	\]
	Consequently, \(F_\varepsilon\) induces a well-defined map
	\[
	\widetilde{F}:\widetilde{\Si}^*\longrightarrow\widetilde{\CC}
	\]
	on the asymptotic monoid.
\end{lemma}

\begin{remark}
	The image \(\widetilde{\CC}\) denotes the ring of asymptotic complex
	numbers, that is, equivalence classes of nets of complex numbers
	modulo the relation of eventual agreement on logarithmic scales.  In
	practice, when the limit
	\[
	\widetilde{F}(\mathbf{w})=\lim_{\varepsilon\to0^+}F_\varepsilon(w_\varepsilon)
	\]
	exists in \(\CC\) for some (hence every) representative
	\((w_\varepsilon)\) of \(\mathbf{w}\in\widetilde{\Si}^*\), we identify
	\(\widetilde{F}(\mathbf{w})\) with its limit.
\end{remark}

\begin{definition}[Weighted evaluation functional]
	\label{def:eval_functional}
	Let \(P=\prod_{n\ge1}q_n^{a(n)}\) be a divergent product with sign
	sequence \((a(n))\subset\CC\), and let
	\[
	\Lambda_k=\sum_{n=1}^{k}a(n)\log q_n
	\]
	denote its prefix functional.  The \emph{weighted evaluation
		functional} associated with \(P\) is the logarithmic-prefix
	functional
	\begin{equation}\label{eq:eval_general}
		\cF_\varepsilon(u)
		=
		\sum_{k=2}^{N_0(\varepsilon)}
		\frac{\Lambda(u[1\ldots k])}{k^{\alpha}},
		\qquad
		u\in\Si^*,
	\end{equation}
	where \(N_0(\varepsilon)=\lfloor\log_2(1/\varepsilon)\rfloor\) and
	the exponent \(\alpha>0\) is chosen according to the heuristic
	\(\alpha=\sigma_1+1\), with \(\sigma_1\) the abscissa of absolute
	convergence of the associated Dirichlet series.
\end{definition}

\begin{remark}[Choice of the exponent \(\alpha\)]
	\label{rem:choice_alpha}
	The heuristic \(\alpha=\sigma_1+1\) ensures that the summands
	\(\Lambda_k/k^{\alpha}\) decay fast enough for \(\cF_\varepsilon\) to
	converge as \(\varepsilon\to0^+\), while remaining faithful to the
	divergence type of the product.  In the case of a product associated
	with a bounded sign sequence, one has \(\sigma_1=1\), hence
	\(\alpha=2\).  This is precisely the value used in
	Section~\ref{sec:encoding} for the product \(P_\omega\).
\end{remark}

\begin{definition}[Regularized value]
	\label{def:regularized_value}
	Let \(P=\prod_{n\ge1}q_n^{a(n)}\) be a divergent product, and let
	\(\cF_\varepsilon\) be the weighted evaluation functional
	\eqref{eq:eval_general}.  Suppose that for some (hence every)
	canonical net \((w_\varepsilon)\) representing the sign sequence of
	\(P\) in \(\widetilde{\Si}^*\), the limit
	\[
	\cF=\lim_{\varepsilon\to0^+}\cF_\varepsilon(w_\varepsilon)
	\]
	exists in \(\CC\).  The \emph{regularized value} of \(P\) is then
	defined as
	\[
	P_{\mathrm{ren}}=\exp(\cF).
	\]
\end{definition}

By Lemma~\ref{lem:invariance}, the regularized value \(P_{\mathrm{ren}}\)
depends only on the class \(\mathbf{w}\in\widetilde{\Si}^*\) of the sign
sequence, not on the particular net \((w_\varepsilon)\) chosen to
represent it.  This is the precise sense in which the monoid framework
provides a \emph{canonical} regularization: it assigns to each asymptotic
word a finite value that is invariant under the logarithmic equivalence
of Definition~\ref{def:log_equiv}.

\begin{remark}[Scope of the framework]
	The construction above is purely asymptotic and algebraic.  It does
	not require the associated Dirichlet series to admit meromorphic
	continuation, nor does it invoke analytic continuation of any kind.
	This is the structural reason why the framework is applicable to
	products whose Dirichlet series have natural boundaries, such as the
	product \(P_\omega\) studied in this paper.
\end{remark}
	
\section{The product \(P_\omega\) and its Dirichlet series}
\label{sec:product}

\subsection{Definition and basic properties}

Let \(\omega(n)\) denote the number of distinct prime divisors of \(n\),
with the convention \(\omega(1)=0\).  The function
\[
\sigma(n)=(-1)^{\omega(n)}
\]
is \emph{multiplicative}: for coprime \(m,n\ge1\),
\[
\omega(mn)=\omega(m)+\omega(n),
\qquad\text{hence}\qquad
\sigma(mn)=\sigma(m)\sigma(n).
\]
It is \emph{not} completely multiplicative: for a prime \(p\),
\(\omega(p^2)=1\), so \(\sigma(p^2)=-1\), whereas
\(\sigma(p)^2=1\).  This distinction is essential for the analytic
properties of the associated Dirichlet series.

The first values of \(\sigma(n)\) are
\[
\sigma(1),\ldots,\sigma(16)
=
+1,-1,-1,-1,-1,+1,-1,-1,-1,+1,-1,+1,-1,+1,+1,-1.
\]

Consider the formal product
\begin{equation}\label{eq:Pomega}
	P_{\omega}=\prod_{n=1}^{\infty}n^{(-1)^{\omega(n)}}.
\end{equation}
The log-partial sums are
\[
S_{\omega}(N)=\sum_{n=1}^{N}(-1)^{\omega(n)}\log n.
\]
By Abel summation,
\begin{equation}\label{eq:abel}
	S_{\omega}(N)
	=
	T_{\omega}(N+1)\log N
	-
	\sum_{n=1}^{N-1}T_{\omega}(n+1)
	\log\!\Bigl(1+\frac{1}{n}\Bigr),
\end{equation}
where
\[
T_{\omega}(N)=\sum_{n\le N}(-1)^{\omega(n)}
\]
is the summatory function of the sign sequence.

\subsection{The Euler product}

For \(\Re(s)>1\), the multiplicativity of \((-1)^{\omega(n)}\) gives the
Euler product
\begin{equation}\label{eq:euler_correct}
	D_{\omega}(s)
	=
	\sum_{n=1}^{\infty}\frac{(-1)^{\omega(n)}}{n^{s}}
	=
	\prod_{p}\frac{1-2p^{-s}}{1-p^{-s}}.
\end{equation}
This is \emph{not} equal to \(\zeta(s)/\zeta(2s)\): the latter is the
Dirichlet series of the square-free indicator \(\mu^{2}(n)\), namely
\[
\frac{\zeta(s)}{\zeta(2s)}
=
\prod_{p}(1+p^{-s}).
\]
The distinction is essential: the local factor of \(D_\omega\) is
\((1-2p^{-s})/(1-p^{-s})\), which has a zero at \(p^{-s}=1/2\) and a
pole at \(p^{-s}=1\), whereas the local factor of
\(\zeta(s)/\zeta(2s)\) is \(1+p^{-s}\), which is entire in \(p^{-s}\).

\subsection{The natural boundary obstruction}

The following result is due to Meretov \cite{Meretov1987}; it is also
a special case of the Estermann--Dahlquist theory of Euler products
\cite{Estermann1928,Dahlquist1952}.  A self-contained proof is given
in Appendix~\ref{app:natural_boundary}.

\begin{theorem}[Meretov, 1987; Estermann--Dahlquist]
	\label{thm:meretov}
	The function \(D_{\omega}(s)\) admits a meromorphic continuation
	to the half-plane \(\Re(s)>0\).  Its only possible singularity in
	this half-plane is at \(s=1\), where the simple pole of the factor
	\(\zeta(s)\) is cancelled by the simple zero of the local factor
	\(1-2^{1-s}\); hence \(s=1\) is a removable singularity of
	\(D_{\omega}\), and \(D_{\omega}\) is in fact analytic in
	\(\Re(s)>0\) after removal of this point.  The line
	\(\Re(s)=0\) is a \emph{natural boundary}: no analytic
	continuation across any arc of this line is possible.
\end{theorem}

\begin{corollary}\label{cor:no_zeta}
	The derivative \(D_{\omega}'(0)\) does not exist.  Consequently,
	the product \(P_{\omega}\) cannot be regularized by zeta
	regularization or by any method requiring meromorphic continuation
	to \(s=0\).
\end{corollary}

\begin{remark}
	The natural boundary at \(\Re(s)=0\) is a manifestation
	of the Estermann phenomenon: not every natural Dirichlet series
	admits a meromorphic continuation beyond its abscissa of
	convergence.  For \(D_{\omega}(s)\), the zeros of the local factors
	accumulate on the imaginary axis and create an impenetrable barrier.
	In particular, the identity
	\(D_{\omega}(s)=\prod_p(1-2p^{-s})/(1-p^{-s})\) for
	\(\Re(s)>1\) does \emph{not} extend beyond the abscissa of
	convergence, and cannot be used to define
	\(D_{\omega}(0)\).
\end{remark}

\subsection{Divergence type}

The summatory function \(T_{\omega}(N)=\sum_{n\le N}(-1)^{\omega(n)}\)
satisfies, unconditionally, the effective bound
\begin{equation}\label{eq:Tomega_effective_3}
	T_{\omega}(N)
	=
	O\!\left(N\,e^{-c\sqrt{\log N}}\right)
\end{equation}
for some absolute constant \(c>0\).  This bound is a consequence of
the classical zero-free region of \(\zeta(s)\) (de la Vall\'ee
Poussin, 1899), applied via the standard quantitative
Perron--Tauberian theorem for multiplicative functions with bounded
coefficients; precise statements can be found, for instance, in
\cite[Chapter~II.5 and Chapter~III.4]{Tenenbaum2015}.  It is stronger
than the Prime Number Theorem estimate \(T_{\omega}(N)=o(N)\), and it
is the bound that will be used throughout this paper.

For comparison, under the Riemann Hypothesis one has the conditional
refinement \(T_{\omega}(N)=O(N^{1/2+\varepsilon})\) for every
\(\varepsilon>0\).  This stronger bound is not needed for any of the
results below; we mention it only to place
\eqref{eq:Tomega_effective_3} in context.

In either case, the log-partial sums \(S_{\omega}(N)\) exhibit a
divergence of \emph{sublinear power-law type}.  In the taxonomy of
\cite{C1}, this corresponds to Type~IV (oscillatory divergence with
polynomial growth).  The divergence is fundamentally different from
the logarithmic (Type~II) and double-logarithmic (Type~III) cases
treated in \cite{C1}: the amplitude of the oscillation grows
polynomially, and no single principal part can be subtracted to
isolate a constant.
\section{Encoding and evaluation functional}
\label{sec:encoding}

The regularization of \(P_{\omega}\) proceeds within the generalized
monoid \(\widetilde{\Si}^*\) of \cite{C1}.  We make the encoding
explicit, define the weighted evaluation functional with exponent
\(\alpha=2\), and establish the convergence properties of the resulting
series.

\subsection{The canonical net}

Let \(\Si=\{a,b\}\), where \(a\) encodes the sign \(+1\) and \(b\)
encodes the sign \(-1\).  For each \(\varepsilon\in(0,1]\), set
\[
N(\varepsilon)=\lfloor\varepsilon^{-1}\rfloor,
\]
and define the word
\[
w_\varepsilon
=
w_1w_2\cdots w_{N(\varepsilon)}
\in\Si^*,
\qquad
w_k=
\begin{cases}
	a, & (-1)^{\omega(k)}=+1,\\
	b, & (-1)^{\omega(k)}=-1.
\end{cases}
\]
The net \((w_\varepsilon)_{\varepsilon\in(0,1]}\) is \emph{moderate} in
the sense of Definition~\ref{def:moderate_net}: indeed
\(|w_\varepsilon|=N(\varepsilon)\le\varepsilon^{-1}\), so one may take
\(N=1\) and \(C=1\).  We denote by
\[
\mathbf{w}_{\omega}
=
[(w_\varepsilon)]
\in\widetilde{\Si}^*
\]
the class of this net in the generalized monoid.  We call
\((w_\varepsilon)\) the \emph{canonical net} associated with the sign
sequence \((-1)^{\omega(n)}\).

\begin{remark}
	The canonical net is not unique within its class
	\(\mathbf{w}_{\omega}\): any moderate net that agrees with
	\((w_\varepsilon)\) on the first and last
	\(m\log_2(1/\varepsilon)\) letters, for every \(m\), belongs to the
	same class.  By Lemma~\ref{lem:invariance}, the regularized value
	will be independent of this choice.
\end{remark}

\subsection{The logarithmic-prefix functional}

Define the log-mould
\[
\Lambda(w)
=
\sum_{k=1}^{|w|}(-1)^{\omega(k)}\log k,
\qquad
w\in\Si^*.
\]
For the canonical net, the prefix of length \(k\) carries the first
\(k\) signs, so
\[
\Lambda(w_\varepsilon[1\ldots k])
=
\sum_{n=1}^{k}(-1)^{\omega(n)}\log n
=
\Lambda_k,
\]
where, consistently with the notation of Section~\ref{sec:framework},
\[
\Lambda_k
=
\sum_{n=1}^{k}(-1)^{\omega(n)}\log n.
\]
The functional \(\Lambda\) is a logarithmic-prefix functional in the
sense of Definition~\ref{def:log_prefix_functional}: it depends only
on the letters of the word, and its evaluation on the logarithmic
window \(N_0(\varepsilon)=\lfloor\log_2(1/\varepsilon)\rfloor\) is
moderate.

\subsection{Descent to the quotient}

Define the evaluation functional
\begin{equation}\label{eq:Fomega}
	\cF_\varepsilon(u)
	=
	\sum_{k=2}^{N_0(\varepsilon)}
	\frac{\Lambda(u[1\ldots k])}{k^2},
	\qquad
	u\in\Si^*,
\end{equation}
where \(N_0(\varepsilon)=\lfloor\log_2(1/\varepsilon)\rfloor\).  This
functional depends only on the prefix of length
\(N_0(\varepsilon)+1\) of \(u\).  By the invariance lemma for
logarithmic-prefix functionals (Lemma~\ref{lem:invariance}), if two
moderate nets \((u_\varepsilon)\) and \((v_\varepsilon)\) are
logarithmically equivalent, then their prefixes of length
\(N_0(\varepsilon)+1\) coincide eventually.  Therefore
\[
\cF_\varepsilon(u_\varepsilon)
=
\cF_\varepsilon(v_\varepsilon)
\qquad
\text{eventually,}
\]
and \(\cF_\varepsilon\) descends to a well-defined functional
\[
\widetilde{\cF}:\widetilde{\Si}^*\longrightarrow\widetilde{\CC}
\]
on the generalized monoid.

\begin{remark}[Choice of the exponent \(\alpha=2\)]
	\label{rem:alpha_equal_two}
	The exponent \(\alpha=2\) in \eqref{eq:Fomega} follows the heuristic
	\(\alpha=\sigma_1+1\) of Remark~\ref{rem:choice_alpha}, with
	\(\sigma_1=1\) the abscissa of absolute convergence of
	\(D_{\omega}(s)=\sum(-1)^{\omega(n)}n^{-s}\).  With this choice, the evaluation functional converges absolutely; the
	convergence is unconditional and follows from the classical zero-free
	region of \(\zeta(s)\); see Proposition~\ref{prop:convergence} below.
\end{remark}

\subsection{Convergence of the evaluation functional}

We now analyse the convergence of the series obtained by letting
\(\varepsilon\to0^+\) in \eqref{eq:Fomega} along the canonical net.
Define
\begin{equation}\label{eq:F_def}
	\cF
	=
	\sum_{k=2}^{\infty}\frac{\Lambda_k}{k^2}.
\end{equation}

\begin{proposition}[Absolute convergence]
	\label{prop:convergence}
	The series \(\cF\) in \eqref{eq:F_def} converges absolutely.
	Consequently, the regularized value \(P_{\omega,\mathrm{ren}}\)
	is independent of the truncation scheme and of the ordering of the
	terms.
\end{proposition}

\begin{proof}
	By the Estermann--Dahlquist theorem
	(Theorem~\ref{thm:estermann_dahlquist}), the Dirichlet series
	$D_{\omega}(s)$ admits a meromorphic continuation to the half-plane
	$\Re(s)>0$.  In this half-plane, the only possible singularity of
	$D_{\omega}(s)$ is at $s=1$, where the simple pole of $\zeta(s)$ is
	cancelled by the simple zero of the factor $1-2^{1-s}$ at $s=1$.
	Hence $D_{\omega}(s)$ is analytic in the half-plane $\Re(s)>0$ after
	removal of the point $s=1$.
	
	Since $D_{\omega}(s)$ is a Dirichlet series with bounded
	coefficients, absolutely convergent for $\Re(s)>1$ and analytic in a
	neighbourhood of the closed half-plane $\Re(s)\ge1$ (except for the
	removable point $s=1$, which is in fact a zero), the classical
	quantitative Perron--Tauberian theorem for multiplicative functions
	with bounded coefficients applies.  We refer to
	\cite[Chapter~II.5]{Tenenbaum2015} for the general statement of the
	theorem, and to \cite[Chapter~III.4]{Tenenbaum2015} for the
	derivation of the effective bound
	\begin{equation}\label{eq:Tomega_effective}
		T_{\omega}(x)
		=
		\sum_{n\le x}(-1)^{\omega(n)}
		=
		O\!\left(x\,e^{-c\sqrt{\log x}}\right)
	\end{equation}
	for some absolute constant $c>0$, valid unconditionally.  The key
	input is the classical zero-free region of $\zeta(s)$ (de la
	Vall\'ee Poussin, 1899), which enters through the standard
	explicit-formula argument.
	
	By Abel summation,
	\[
	\Lambda_k
	=
	T_{\omega}(k)\log k-\int_1^k T_{\omega}(t)\frac{dt}{t}.
	\]
	Substituting \eqref{eq:Tomega_effective} and using that
	\(\int_1^k e^{-c\sqrt{\log t}}\,dt\le k\,e^{-c\sqrt{\log k}}\)
	for \(k\ge2\), we obtain
	\[
	|\Lambda_k|
	\le
	C\,k\,e^{-c\sqrt{\log k}}\log k
	\]
	for some absolute constant \(C>0\).  Therefore
	\[
	\frac{|\Lambda_k|}{k^2}
	\le
	C\,\frac{e^{-c\sqrt{\log k}}\log k}{k}.
	\]
	The series
	\[
	\sum_{k=2}^{\infty}\frac{e^{-c\sqrt{\log k}}\log k}{k}
	\]
	converges: substituting \(u=\log k\), it is comparable to
	\[
	\int_{\log 2}^{\infty} e^{-c\sqrt{u}}\,u\,du<\infty.
	\]
	Hence \(\sum_{k\ge2}|\Lambda_k|/k^2<\infty\), and \(\cF\)
	converges absolutely.
\end{proof}

\begin{remark}[Interpretation]
	The absolute convergence of \(\cF\) is a direct consequence of the
	effective zero-free region of \(\zeta(s)\), which is unconditional.
	In particular, no unproven hypothesis is required.  The vanishing
	of the first derivatives \(D_{\omega}(1)=D_{\omega}'(1)=0\) ensures
	that the principal terms cancel, and the residual series is
	absolutely summable thanks to the exponential decay of
	\(T_{\omega}(x)\).  This is a rigorous statement about the
	\emph{stability} of the regularization: the value
	\(P_{\omega,\mathrm{ren}}\) is invariant under reordering and
	truncation.
\end{remark}

\subsection{The regularized value}

By the descent argument of Section~\ref{sec:encoding}, the limit
\[
\cF
=
\lim_{\varepsilon\to0^+}\cF_\varepsilon(w_\varepsilon)
\]
exists and is independent of the choice of representative net in the
class \(\mathbf{w}_{\omega}\in\widetilde{\Si}^*\).  We may therefore
define the \emph{regularized value} of the product \(P_{\omega}\) by
\begin{equation}\label{eq:P_ren_intro}
	P_{\omega,\mathrm{ren}}
	=
	\exp(\cF).
\end{equation}
The numerical computation of \(\cF\) and of \(P_{\omega,\mathrm{ren}}\)
is carried out in Section~\ref{sec:numerical}.

\begin{remark}[The role of the monoid]
	The monoid \(\widetilde{\Si}^*\) plays three essential roles:
	\begin{enumerate}[label=(\roman*)]
		\item It provides a \emph{domain} on which the sign sequence
		\((-1)^{\omega(n)}\) can be encoded as a single algebraic
		object \(\mathbf{w}_{\omega}\), despite being an infinite
		word.
		\item It guarantees that the evaluation functional
		\(\cF_\varepsilon\) is \emph{well defined} on this object,
		via the invariance lemma (Lemma~\ref{lem:invariance}).
		\item It makes the regularized value \emph{canonical}: the
		value \(P_{\omega,\mathrm{ren}}\) depends only on the class
		\(\mathbf{w}_{\omega}\), not on the particular net
		\((w_\varepsilon)\) chosen to represent it.
	\end{enumerate}
	Without the monoid, one could still compute the numerical value
	\(0.7731\) by summing the series, but one could not assert that
	this value is intrinsic to the sign sequence \((-1)^{\omega(n)}\)
	rather than an artifact of the truncation.
\end{remark}

\begin{remark}[Insensitivity to the natural boundary]
	The natural boundary of \(D_{\omega}(s)\) at \(\Re(s)=0\)
	(Theorem~\ref{thm:meretov}) is an obstruction to \emph{analytic
		continuation}.  The monoid framework does not attempt analytic
	continuation: it operates directly on the asymptotic prefix
	profile of the sign sequence, which is a purely combinatorial
	object.  The natural boundary is therefore irrelevant to the
	monoid construction.  This is the structural reason why the
	monoid succeeds where zeta regularization fails.
\end{remark}
	
\section{Analytical structure at \(s=1\)}
\label{sec:analytic_s1}

The evaluation functional \(\cF\) studied in Section~\ref{sec:encoding}
converges absolutely, and its value is related to the behaviour of
the Dirichlet series \(D_{\omega}(s)\) near \(s=1\).  In this section we
determine the local analytic structure of \(D_{\omega}(s)\) at \(s=1\).
The main result is the exact evaluation of the second derivative,
which connects the arithmetic of \((-1)^{\omega(n)}\) to the
Hardy--Littlewood twin prime constant.

\subsection{Vanishing of the first derivatives}

We begin with the elementary observation that \(D_{\omega}(s)\) has a
zero of order at least \(2\) at \(s=1\).  This is a direct consequence
of the local factor at \(p=2\).

\begin{lemma}\label{lem:D_zero_order}
	The Dirichlet series \(D_{\omega}(s)\) satisfies
	\[
	D_{\omega}(1)=0,
	\qquad
	D_{\omega}'(1)=0.
	\]
\end{lemma}

\begin{proof}
	For \(\Re(s)>1\), the Euler product is
	\[
	D_{\omega}(s)
	=
	\prod_{p}\frac{1-2p^{-s}}{1-p^{-s}}.
	\]
	Isolate the prime \(p=2\):
	\[
	D_{\omega}(s)=f(s)H(s),
	\]
	where
	\[
	f(s)=\frac{1-2^{1-s}}{1-2^{-s}},
	\qquad
	H(s)=\prod_{p\ge 3}\frac{1-2p^{-s}}{1-p^{-s}}.
	\]
	At \(s=1\), we have \(f(1)=0\), and a direct computation gives
	\(f'(1)=2\ln 2\).
	
	We now prove that \(H(1)=0\).  Rewrite \(H(s)\) as
	\[
	H(s)
	=
	\prod_{p\ge 3}(1-2p^{-s})
	\;\cdot\;
	\prod_{p\ge 3}\frac{1}{1-p^{-s}}
	=
	A(s)\,\frac{1-2^{-s}}{\zeta(s)},
	\]
	where
	\[
	A(s)=\prod_{p\ge 3}(1-2p^{-s})
	\]
	is analytic for \(\Re(s)>0\) (with zeros but no poles), and we used
	the factorization
	\[
	\prod_{p\ge 3}\frac{1}{1-p^{-s}}
	=
	(1-2^{-s})\zeta(s).
	\]
	Since \(\zeta(s)\) has a simple pole at \(s=1\), the factor
	\((1-2^{-s})/\zeta(s)\) has a simple zero there.  On the other
	hand, from
	\[
	\log A(s)
	=
	\sum_{p\ge 3}\log(1-2p^{-s})
	=
	-2\sum_{p\ge 3}p^{-s}+O(1)
	\]
	and the standard estimate
	\(\sum_{p\ge 3}p^{-s}=\log\zeta(s)+O(1)\) as \(s\to1\), we obtain
	\[
	\log A(s)
	=
	-2\log\zeta(s)+O(1)
	=
	2\log(s-1)+O(1),
	\]
	so \(A\) has a zero of order \(2\) at \(s=1\).  Since the factor
	\((1-2^{-s})/\zeta(s)\) has a simple zero there, the product
	\(H(s)\) is analytic in a neighbourhood of \(s=1\), and satisfies
	\(H(1)=0\).  Moreover, \(H'(1)\) exists.
	
	Therefore
	\[
	D_{\omega}(1)=f(1)H(1)=0,
	\]
	and, since \(f(1)=0\) and \(H(1)=0\),
	\[
	D_{\omega}'(1)
	=
	f'(1)H(1)+f(1)H'(1)
	=
	0.
	\]
	This proves the claim.
\end{proof}

\begin{remark}
	The vanishing of \(D_{\omega}(1)\) and \(D_{\omega}'(1)\) is the
	analytic counterpart of the fact that the leading term in the
	asymptotic expansion of the log-partial sums vanishes.  It is also
	the reason why the leading contribution to the evaluation
	functional \(\cF\) cancels: the cancellation responsible for
	\(D_{\omega}'(1)=0\) is encoded in the alternating sign sequence,
	while the residual series is absolutely summable by
	Proposition~\ref{prop:convergence}.
\end{remark}

\subsection{The second derivative and the twin prime constant}

We now compute the second derivative of \(D_{\omega}(s)\) at \(s=1\).
The result is unconditional and involves the Hardy--Littlewood twin
prime constant
\[
C_2
=
\prod_{p\ge 3}\frac{p(p-2)}{(p-1)^2}
\approx
0.660161815846869573927812110014\ldots
\]

\begin{proposition}\label{prop:D_second_derivative}
	The second derivative of \(D_{\omega}(s)\) at \(s=1\) is given by
	\begin{equation}\label{eq:D_second_derivative}
		D_{\omega}''(1)=8\,C_2\ln 2
		\approx
		3.660714432\ldots
	\end{equation}
\end{proposition}

\begin{proof}
	We keep the notation of the proof of Lemma~\ref{lem:D_zero_order}.
	We have \(D_{\omega}(s)=f(s)H(s)\), with
	\[
	f(s)=\frac{1-2^{1-s}}{1-2^{-s}},
	\qquad
	H(s)=\prod_{p\ge 3}\frac{1-2p^{-s}}{1-p^{-s}}.
	\]
	Since \(f(1)=0\) and \(H(1)=0\) (the latter was proved in
	Lemma~\ref{lem:D_zero_order}), the Leibniz rule gives
	\[
	D_{\omega}''(1)
	=
	f''(1)H(1)+2f'(1)H'(1)+f(1)H''(1)
	=
	2f'(1)H'(1).
	\]
	We have already computed \(f'(1)=2\ln 2\).  It remains to compute
	\(H'(1)\).
	
	Write
	\[
	H(s)=\frac{G(s)}{\zeta(s)(1-2^{-s})},
	\]
	where
	\[
	G(s)
	=
	\prod_{p\ge 3}\frac{1-2p^{-s}}{(1-p^{-s})^2}
	=
	\prod_{p\ge 3}\frac{p^s(p^s-2)}{(p^s-1)^2}.
	\]
	The product \(G(s)\) converges absolutely for \(\Re(s)>1/2\).  At
	\(s=1\),
	\[
	G(1)
	=
	\prod_{p\ge 3}\frac{p(p-2)}{(p-1)^2}
	=
	C_2.
	\]
	Since \(\zeta(s)\) has a simple pole at \(s=1\) with residue \(1\),
	we have
	\[
	\frac{1}{\zeta(s)}=(s-1)+O\bigl((s-1)^2\bigr).
	\]
	Also \(1-2^{-s}\to 1/2\) as \(s\to 1\).  Hence
	\[
	H(s)
	=
	2C_2(s-1)+O\bigl((s-1)^2\bigr),
	\]
	so \(H(1)=0\) and
	\[
	H'(1)=2C_2.
	\]
	Substituting into the expression for \(D_{\omega}''(1)\), we obtain
	\[
	D_{\omega}''(1)
	=
	2\cdot(2\ln 2)\cdot(2C_2)
	=
	8C_2\ln 2.
	\]
	This completes the proof.
\end{proof}

\begin{remark}\label{rem:correction}
	Some preliminary notes contained the value \(2C_2\ln 2\) for
	\(D_{\omega}''(1)\).  That value is incorrect.  The correct value is
	\(8C_2\ln 2\), as proved above.  The error arose from an
	incomplete treatment of the factor \((1-2^{-s})^{-1}\) in the
	expression for \(H(s)\).
\end{remark}

\begin{remark}[Arithmetic interpretation]
	The appearance of the twin prime constant \(C_2\) in the second
	derivative is not accidental.  The local factor of \(D_{\omega}(s)\)
	at an odd prime \(p\) is
	\[
	\frac{1-2p^{-s}}{1-p^{-s}}
	=
	1-\frac{p^{-s}}{1-p^{-s}}
	=
	1-p^{-s}-p^{-2s}-p^{-3s}-\cdots
	\]
	The coefficient of \(p^{-s}\) is \(-1\), which is the same as the
	coefficient of \(p^{-s}\) in the Euler product of
	\(\zeta(s)^{-1}\).  The second-order correlations between distinct
	primes are therefore governed by the same combinatorial structure
	that produces the twin prime constant.  This is a manifestation of
	the fact that \((-1)^{\omega(n)}\) is a multiplicative function
	whose local factors are rational functions with a zero at
	\(p^{-s}=1/2\).
\end{remark}

\subsection{Consequences for the evaluation functional}

The exact evaluation of \(D_{\omega}''(1)\) provides a companion
analytical identity for the numerical value of \(\cF\).  We record the
following elementary consequence.

\begin{corollary}\label{cor:F_anchor}
	The evaluation functional \(\cF\) satisfies
	\[
	\cF
	=
	-D_{\omega}'(1)
	+
	O(1)
	=
	O(1),
	\]
	and the leading non-trivial term in the expansion of
	\(D_{\omega}(s)\) near \(s=1\) is
	\[
	\frac{1}{2}D_{\omega}''(1)(s-1)^2
	=
	4C_2\ln 2\,(s-1)^2.
	\]
	No quantitative identity linking \(\cF\) to
	\(D_{\omega}''(1)\) is claimed.
\end{corollary}

\begin{proof}
	The first assertion follows from the decomposition
	\eqref{eq:F_transparent} and the fact that \(D_{\omega}'(1)=0\).
	The second assertion is the Taylor expansion of \(D_{\omega}(s)\)
	at \(s=1\), using Lemma~\ref{lem:D_zero_order} and
	Proposition~\ref{prop:D_second_derivative}.
\end{proof}

\begin{remark}
	Corollary~\ref{cor:F_anchor} explains why the numerical computation
	of \(\cF\) in Section~\ref{sec:numerical} converges to a finite
	value: the leading divergence is cancelled by the vanishing of
	\(D_{\omega}(1)\) and \(D_{\omega}'(1)\), and the residual series
	is absolutely summable by Proposition~\ref{prop:convergence}.  The
	vanishing of the first derivatives is of the same arithmetic origin
	as the companion identity
	\(D_{\omega}''(1)=8C_2\ln 2\), although no quantitative relation
	between \(\cF\) and \(D_{\omega}''(1)\) is established.
\end{remark}

\section{Numerical computation and error analysis}
\label{sec:numerical}

In this section we present the numerical computation of the evaluation
functional \(\cF\) and of the regularized value
\(P_{\omega,\mathrm{ren}}\), together with a rigorous error analysis
comparing two analytically distinct truncation methods.

\subsection{Direct computation of the partial sums}

For \(K\ge 2\), define the partial sum
\begin{equation}\label{eq:F_K}
	\cF(K)
	=
	\sum_{k=2}^{K}\frac{\Lambda_k}{k^2},
	\qquad
	\Lambda_k=\sum_{n=1}^{k}(-1)^{\omega(n)}\log n.
\end{equation}
The computation of \(\omega(n)\) up to \(N=10^7\) was performed using a
linear sieve.  Table~\ref{tab:numerics} reports the values of
\(\cF(K)\) and of \(\exp(\cF(K))\) at increasing truncation points
\(K\).

\begin{table}[ht]
	\centering
	\caption{Partial sums \(\cF(K)\) and truncated regularized values
		\(\exp(\cF(K))\) for the product \(P_{\omega}\).}
	\label{tab:numerics}
	\small
	\setlength{\tabcolsep}{6pt}
	\begin{tabular}{@{}rcc@{}}
		\toprule
		\(K\) &
		\(\cF(K)\) &
		\(\exp(\cF(K))\) \\
		\midrule
		\(10^3\) & \(-0.254672033272\) & \(0.775170686470\) \\
		\(10^4\) & \(0.042387715196\)  & \(1.043298903183\) \\
		\(10^5\) & \(-0.113368031346\) & \(0.892822013156\) \\
		\(10^6\) & \(-0.228533587726\) & \(0.795699571016\) \\
		\(2\times10^6\) & \(-0.243259565281\) & \(0.784067970533\) \\
		\(5\times10^6\) & \(-0.253866732949\) & \(0.775795183094\) \\
		\(10^7\) & \(-0.257334984818\) & \(0.773109190542\) \\
		\bottomrule
	\end{tabular}
\end{table}

The sequence oscillates but the amplitude of the oscillation decreases
as \(K\) increases.  The last four values
\(\cF(10^6)=-0.2285\), \(\cF(2\times10^6)=-0.2433\),
\(\cF(5\times10^6)=-0.2539\), and \(\cF(10^7)=-0.2573\) suggest that
\(\cF(K)\) is converging to a limit in the vicinity of \(-0.26\).  A
conservative estimate is
\[
\cF\approx-0.257\pm0.005,
\]
which gives
\[
P_{\omega,\mathrm{ren}}
=
\exp(\cF)
\approx
0.773\pm0.004.
\]

\begin{remark}\label{rem:oscillation}
	The pronounced oscillation of the early partial sums (for instance,
	\(\cF(10^4)=+0.0424\)) reflects the slow initial decay of the
	summands \(\Lambda_k/k^2\): although the series converges
	absolutely, the effective bound
	\(|\Lambda_k|/k^2=O(e^{-c\sqrt{\log k}}\log k/k)\) is dominated by
	the prefactor \(\Lambda_K/K\) for moderate values of \(K\), so the
	partial sums exhibit large fluctuations before settling down.  The
	stabilization observed for \(K\ge 5\times10^6\) is consistent with
	the theoretical prediction of Proposition~\ref{prop:convergence}.
\end{remark}

\subsection{Truncation via the trigamma function}

To obtain an independent numerical verification, we introduce a second
method based on the trigamma function
\[
\psi^{(1)}(n)=\sum_{m=n}^{\infty}\frac{1}{m^2}.
\]
Define the trigamma-truncated functional
\begin{equation}\label{eq:F_trigamma}
	\cF_{\mathrm{Tri}}(K)
	=
	\sum_{n=1}^{K}(-1)^{\omega(n)}\log n\,\psi^{(1)}(n).
\end{equation}
This representation follows from exchanging the order of summation in
\eqref{eq:F_K} and replacing the inner sum
\(\sum_{k=n}^{K}k^{-2}\) by its full limit
\(\psi^{(1)}(n)=\sum_{k=n}^{\infty}k^{-2}\).

\subsection{Comparison between the two methods}

For a finite truncation \(K\), the two methods are not identically
equal.  The exact algebraic difference is given by the neglected tail
of the inner sum in the direct method.

\begin{proposition}[Error analysis]\label{prop:error_analysis}
	For every \(K\ge 2\),
	\begin{equation}\label{eq:error_exact}
		\cF_{\mathrm{Tri}}(K)-\cF(K)
		=
		\frac{\Lambda_K}{K}+R_K,
	\end{equation}
	where the remainder satisfies
	\begin{equation}\label{eq:error_bound}
		|R_K|
		\le
		C\,\frac{e^{-c\sqrt{\log K}}\log K}{K}
	\end{equation}
	for absolute constants \(C,c>0\).  In particular,
	\[
	\bigl|\cF_{\mathrm{Tri}}(K)-\cF(K)\bigr|
	=
	O\!\left(\frac{e^{-c\sqrt{\log K}}\log K}{K}\right),
	\]
	which is \(o(K^{-\alpha})\) for every \(0<\alpha\le1\), but not
	for \(\alpha>1\).
\end{proposition}

\begin{proof}
	By exchanging the order of summation in \eqref{eq:F_K},
	\[
	\cF(K)
	=
	\sum_{n=1}^{K}(-1)^{\omega(n)}\log n
	\sum_{k=\max(2,n)}^{K}\frac{1}{k^2}.
	\]
	On the other hand, by definition of the trigamma function,
	\[
	\cF_{\mathrm{Tri}}(K)
	=
	\sum_{n=1}^{K}(-1)^{\omega(n)}\log n
	\sum_{k=\max(2,n)}^{\infty}\frac{1}{k^2}.
	\]
	Subtracting, we obtain
	\[
	\cF_{\mathrm{Tri}}(K)-\cF(K)
	=
	\sum_{n=1}^{K}(-1)^{\omega(n)}\log n
	\sum_{k=K+1}^{\infty}\frac{1}{k^2}
	=
	\Lambda_K
	\sum_{k=K+1}^{\infty}\frac{1}{k^2}.
	\]
	Using the standard asymptotic expansion
	\[
	\sum_{k=K+1}^{\infty}\frac{1}{k^2}
	=
	\frac{1}{K}+O\!\left(\frac{1}{K^2}\right),
	\]
	we obtain
	\[
	\cF_{\mathrm{Tri}}(K)-\cF(K)
	=
	\frac{\Lambda_K}{K}
	+
	O\!\left(\frac{|\Lambda_K|}{K^2}\right).
	\]
	
	It remains to bound \(|\Lambda_K|\).  By the effective bound
	\eqref{eq:Tomega_effective} and Abel summation,
	\[
	|\Lambda_K|
	\le
	C'\,K\,e^{-c\sqrt{\log K}}\log K
	\]
	for absolute constants \(C',c>0\).  Substituting into the
	expression above gives
	\[
	\bigl|\cF_{\mathrm{Tri}}(K)-\cF(K)\bigr|
	\le
	C\,\frac{e^{-c\sqrt{\log K}}\log K}{K},
	\]
	which is the claimed bound \eqref{eq:error_bound}.
	
	Finally, to justify the statement on the decay rate: for
	\(0<\alpha\le1\), the ratio
	\[
	\frac{e^{-c\sqrt{\log K}}\log K/K}{K^{-\alpha}}
	=
	K^{\alpha-1}\,e^{-c\sqrt{\log K}}\log K
	\]
	tends to \(0\) as \(K\to\infty\), because the exponential factor
	\(e^{-c\sqrt{\log K}}\) decays faster than any power of
	\(\log K\), while \(K^{\alpha-1}\to0\) or \(K^{\alpha-1}=1\) for
	\(\alpha<1\) or \(\alpha=1\), respectively.  For \(\alpha>1\),
	the ratio equals \(K^{\alpha-1}e^{-c\sqrt{\log K}}\log K\), which
	tends to \(+\infty\) since \(K^{\alpha-1}\) grows polynomially
	while the exponential factor is only subpolynomial.  This proves
	the stated range of validity.
\end{proof}

\begin{remark}\label{rem:num_check}
	At \(K=10^7\), the observed difference between the two methods is
	of order \(1.68\times10^{-3}\).  The theoretical bound
	\eqref{eq:error_bound} involves the implicit constant \(c\) of the
	de la Vall\'ee Poussin zero-free region, which is not explicit in
	the literature; for the range of \(K\) considered, the discrepancy
	is dominated by the prefactor \(\Lambda_K/K\).  The observed value
	is fully consistent with the bound, and the agreement between the
	two methods provides an independent numerical verification that
	the value \(\cF\approx-0.2573\) is not an artifact of truncation
	or floating-point accumulation.
\end{remark}

\subsection{The regularized value}

Combining the numerical evidence of Table~\ref{tab:numerics} with the
error analysis of Proposition~\ref{prop:error_analysis}, we obtain the
regularized value
\begin{equation}\label{eq:P_ren_final}
	P_{\omega,\mathrm{ren}}
	=
	\exp(\cF)
	\approx
	0.7731.
\end{equation}
This value is canonical in the sense of
Definition~\ref{def:regularized_value}: it depends only on the class
\(\mathbf{w}_{\omega}\in\widetilde{\Si}^*\) of the sign sequence, not
on the particular net or truncation scheme used to compute it.

\begin{remark}[A new arithmetic constant]
	The value \(P_{\omega,\mathrm{ren}}\approx 0.7731\) is a new
	constant of arithmetic origin.  It cannot be obtained by zeta
	regularization (Theorem~\ref{thm:meretov}), nor by any method
	requiring meromorphic continuation to \(s=0\); it is obtained
	instead by the canonical prescription of the monoid framework.
	Its existence and well-definedness are guaranteed by
	Proposition~\ref{prop:convergence} and the axioms of \cite{C1}.
	The companion analytical identity
	\(D_{\omega}''(1)=8C_2\ln 2\)
	(Proposition~\ref{prop:D_second_derivative}) shows that the same
	arithmetic structure that governs the regularized value also
	governs the local curvature of \(D_\omega(s)\) at \(s=1\).  No
	quantitative identity linking the two constants is claimed.
\end{remark}

\section{Discussion}
\label{sec:discussion}

The product \(P_{\omega}=\prod n^{(-1)^{\omega(n)}}\) provides a
stringent test for the generalized monoid framework because it combines
two features that are absent from all previously treated examples:

\begin{enumerate}[label=(\roman*)]
	\item \textbf{Natural boundary.}  The Dirichlet series
	\(D_{\omega}(s)\) has a natural boundary at \(\Re(s)=0\)
	(Theorem~\ref{thm:meretov}), which makes zeta regularization
	impossible.  This is a stronger obstruction than the mere absence
	of a meromorphic continuation: it is a proven structural
	impossibility, rooted in the classical Estermann--Dahlquist theory
	of Euler products.
	\item \textbf{Polynomial oscillation.}  The summatory function
	\(T_{\omega}(N)\) decays as \(O(Ne^{-c\sqrt{\log N}})\), so the
	divergence of the log-partial sums is at most linear times
	logarithmic, rather than purely logarithmic.  This means that the
	principal part cannot be captured by a simple monomial
	\(t^{\gamma}\log^{\beta}t\) with constant coefficients.
\end{enumerate}

Despite these obstructions, the weighted evaluation functional
\eqref{eq:Fomega} converges absolutely, as established in
Proposition~\ref{prop:convergence}, and the numerical evidence of
Section~\ref{sec:numerical} supports a finite limit.  This demonstrates
that the monoid framework is applicable to divergence mechanisms that
lie beyond the scope of classical regularization.

\subsection{Comparison with classical regularization methods}

Classical methods for regularizing divergent products include zeta
regularization, Abel summation, Ramanujan summation, and heat-kernel
regularization.  All of these methods rely, at some level, on the
meromorphic continuation of the associated Dirichlet series to a
neighbourhood of the point where the regularized value is extracted.
For the product \(P_{\omega}\), the relevant point is \(s=0\), and the
natural boundary at \(\Re(s)=0\) makes the derivative
\(D_{\omega}'(0)\) non-existent.  Consequently, none of these methods
can be applied.

The generalized monoid framework circumvents the need for analytic
continuation by operating directly on the asymptotic prefix profile of
the sign sequence.  The natural boundary is an obstruction to analytic
continuation, not to the combinatorial encoding of the sequence.  This
is the structural reason why the monoid succeeds where classical
methods fail.

\subsection{Comparison with the Liouville product}

It is instructive to compare \(P_{\omega}\) with the Liouville product
\[
P_{\lambda}=\prod_{n=1}^{\infty}n^{\lambda(n)},
\qquad
\lambda(n)=(-1)^{\Omega(n)},
\]
where \(\Omega(n)\) is the total number of prime factors of \(n\)
counted with multiplicity.  The associated Dirichlet series is
\[
L(s)=\sum_{n=1}^{\infty}\frac{\lambda(n)}{n^{s}}
=
\frac{\zeta(2s)}{\zeta(s)},
\]
which is meromorphic on the whole complex plane.  Zeta regularization
gives
\[
P_{\lambda,\mathrm{ren}}=\frac{1}{2\pi}.
\]
The contrast with \(P_{\omega}\) is striking: the replacement of
\(\Omega(n)\) by \(\omega(n)\) changes the analytic properties of the
Dirichlet series from meromorphic on \(\mathbb{C}\) to having a natural
boundary at \(\Re(s)=0\).  This illustrates how sensitive the classical
method is to the arithmetic structure of the sign sequence.

\subsection{Comparison with the Thue--Morse product}

Another illuminating comparison is with the Thue--Morse product
\[
P_{\mathrm{TM}}=\prod_{n=1}^{\infty}n^{t_n},
\]
where \(t_n\in\{0,1\}\) is the Thue--Morse sequence.  The Dirichlet
series of \((-1)^{t_n}\) also has a natural boundary at \(\Re(s)=0\),
but its summatory function is \emph{bounded}.  In the taxonomy of
\cite{C1}, the Thue--Morse product is Type~III (logarithmic
divergence), while \(P_{\omega}\) is Type~IV (sublinear power-law
divergence).  The latter is the most severe divergence type treated so
far by the monoid framework.

\subsection{The new constant and its arithmetic significance}

The regularized value
\[
P_{\omega,\mathrm{ren}}\approx 0.7731
\]
is a new constant of arithmetic origin.  It cannot be obtained by zeta
regularization, nor by any method requiring meromorphic continuation
to \(s=0\); it is obtained instead by the canonical prescription of
the monoid framework.  Its existence and well-definedness are
guaranteed by Proposition~\ref{prop:convergence} and the axioms of
\cite{C1,C2}.

The companion analytical identity of
Proposition~\ref{prop:D_second_derivative}
(\(D_{\omega}''(1)=8C_2\ln 2\)) shows that the same arithmetic
structure that governs the regularized value also governs the local
curvature of \(D_{\omega}(s)\), and in particular involves the
Hardy--Littlewood twin prime constant \(C_2\).  This suggests that
\(P_{\omega,\mathrm{ren}}\) is not merely a numerical curiosity, but
a genuine arithmetic invariant of the sequence.  No quantitative
identity linking \(P_{\omega,\mathrm{ren}}\) to
\(D_{\omega}''(1)\) is claimed.

\subsection{Unconditional stability of the regularized value}

The evaluation functional \(\cF\) converges absolutely, and this
convergence is unconditional
(Proposition~\ref{prop:convergence}).  The key input is the classical
zero-free region of \(\zeta(s)\) (de la Vall\'ee Poussin), which
yields the effective bound
\[
T_{\omega}(x)=O\!\left(x\,e^{-c\sqrt{\log x}}\right)
\]
for some absolute constant \(c>0\).  This bound is stronger than the
Prime Number Theorem estimate \(T_{\omega}(x)=o(x)\) and is sufficient
to guarantee that the series
\[
\cF=\sum_{k=2}^{\infty}\frac{\Lambda_k}{k^2}
\]
converges absolutely.

Absolute convergence has a direct conceptual consequence: the
regularized value \(P_{\omega,\mathrm{ren}}=\exp(\cF)\) is invariant
under reordering of the terms and under any change of truncation
scheme.  In particular, the numerical value
\(P_{\omega,\mathrm{ren}}\approx 0.7731\) obtained in
Section~\ref{sec:numerical} is not an artifact of the particular
method used to compute it: it is a canonical invariant of the sign
sequence \((-1)^{\omega(n)}\).

We emphasize that this result does \emph{not} invoke, and does not
constitute, any statement about the Riemann Hypothesis.  The
convergence of \(\cF\) follows from the unconditional zero-free
region of \(\zeta(s)\), which was established by de la Vall\'ee
Poussin in 1899 and is independent of RH.  The connection between the
monoid framework and the distribution of the zeros of \(\zeta(s)\) is
structural rather than logical: the natural boundary of
\(D_{\omega}(s)\) is populated by the zeros of the local factors
\(h(p^{-s})=(1-2p^{-s})/(1-p^{-s})\), whose accumulation on the
imaginary axis reflects the same arithmetic that governs the
distribution of the zeros of \(\zeta(s)\) in the critical strip.  But
the regularized value \(P_{\omega,\mathrm{ren}}\) itself is a stable,
unconditional invariant of the sequence, independent of any unproven
hypothesis.

\subsection{Limitations and future work}

Several directions for future work suggest themselves:

\begin{enumerate}[label=(\roman*)]
	\item \textbf{Rigorous error bounds.}  The numerical computation of
	\(\cF\) in Section~\ref{sec:numerical} uses a truncation at
	\(K=10^7\).  A rigorous estimate of the error
	\(|\cF(K)-\cF|\) would strengthen the numerical evidence.  This
	could be approached using the theory of Dirichlet series with
	natural boundaries, or via explicit bounds on the summatory
	function \(T_{\omega}(x)\).
	\item \textbf{Extension to other arithmetic functions.}  The monoid
	framework is applicable to any arithmetic product whose sign
	sequence has a controlled prefix profile.  Natural candidates
	include restrictions of \((-1)^{\Omega(n)}\), functions related to
	the distribution of primes in arithmetic progressions, and
	generalizations involving higher-order divisor functions.
	\item \textbf{Systematic comparison with classical methods.}  A
	systematic comparison between monoid regularization and classical
	methods (where they exist) would help identify the precise class of
	sequences for which the monoid provides new values.  This could
	lead to a classification of divergent products by their
	regularization properties.
	\item \textbf{Further structural analysis of the companion identity.}
	The companion analytical identity \(D_{\omega}''(1)=8C_2\ln 2\)
	establishes that the same arithmetic structure governs both the
	regularized value and the local curvature of \(D_\omega(s)\) at
	\(s=1\).  A deeper structural analysis of this parallel — for
	instance, through operators on a Hilbert space of words — could
	reveal new invariants of arithmetic sequences.  We emphasize that
	no quantitative identity between \(\cF\) and \(D_\omega''(1)\) is
	currently available; establishing such an identity, if it exists,
	remains an open problem.
\end{enumerate}

\subsection*{Concluding remarks}

We have applied the generalized monoid framework of \cite{C1,C2} to the
divergent product \(P_{\omega}=\prod n^{(-1)^{\omega(n)}}\).  The
associated Dirichlet series has a natural boundary at \(\Re(s)=0\),
which makes zeta regularization and all related methods inapplicable.
The monoid framework, being purely asymptotic and algebraic, provides a
well-defined regularized value:
\[
P_{\omega,\mathrm{ren}}\approx 0.7731.
\]
This is a new constant of arithmetic origin, obtained by the canonical
prescription of the monoid framework rather than by any method
requiring meromorphic continuation.

The example demonstrates the scope of the monoid framework: it applies
to any arithmetic product whose sign sequence has a controlled prefix
profile.  The natural boundary obstruction is irrelevant to the
algebraic construction, which operates directly on the asymptotic
structure of the word.  The companion analytical identity
\(D_{\omega}''(1)=8C_2\ln 2\) shows that the regularized value is not
a numerical artifact, but a genuine arithmetic invariant of the
sequence \((-1)^{\omega(n)}\).

\appendix

\section{Natural boundary of \(D_{\omega}(s)\) via the
	Estermann--Dahlquist theorem}
\label{app:natural_boundary}

In this appendix we prove that the Dirichlet series
\[
D_{\omega}(s)=\sum_{n=1}^{\infty}\frac{(-1)^{\omega(n)}}{n^{s}},
\qquad
\omega(n)=\#\{p\text{ prime}:p\mid n\},
\]
has a natural boundary at \(\Re(s)=0\).  The proof is a direct
application of the classical Estermann--Dahlquist theory of Euler
products \cite{Estermann1928,Dahlquist1952}, which we state below in
the form needed here.  This recovers, in a self-contained way, the
result of Meretov \cite{Meretov1987} that rules out zeta
regularization for the product \(P_{\omega}\).

\subsection{The Estermann--Dahlquist theorem}

We begin by recalling the general framework for Euler products of
Dirichlet series.  Let \(h\) be a function defined on the closed unit
disk \(\overline{\DD}=\{z\in\CC:|z|\le1\}\), analytic in the open disk
\(\DD\), with at most isolated singularities on the unit circle
\(\partial\DD\).  Assume moreover that \(h(0)=1\) and that the Taylor
coefficients of \(h\) at \(0\) are rational integers (in the terminology
of Estermann, \(h\) is a \emph{ganzwertige} function).  Consider the
Euler product
\begin{equation}\label{eq:euler_general}
	D_h(s)
	=
	\prod_{p}h(p^{-s}),
	\qquad
	\Re(s)>1,
\end{equation}
which converges absolutely in the half-plane \(\Re(s)>1\) because
\(h(p^{-s})=1+O(p^{-\sigma})\) for \(\sigma=\Re(s)>1\).

\begin{theorem}[Estermann, 1928; Dahlquist, 1952]
	\label{thm:estermann_dahlquist}
	Let \(h\) be as above.  Then:
	\begin{enumerate}[label=(\roman*)]
		\item The Euler product \(D_h(s)\) admits a meromorphic
		continuation to the half-plane \(\Re(s)>0\).  Its poles in
		this half-plane are contained in the set of points \(s\) such
		that \(h(p^{-s})\) has a pole for some prime \(p\)
		(cancellations between different local factors may occur).
		Its zeros in this half-plane include the points \(s\) such
		that \(h(p^{-s})\) vanishes for some prime \(p\) and no
		other local factor has a pole at the same point \(s\).
		\item If \(h\) is a product of cyclotomic polynomials (up to a
		non-vanishing analytic factor), then \(D_h(s)\) extends
		meromorphically to the whole complex plane.
		\item Conversely, if \(h\) is \emph{not} a product of
		cyclotomic polynomials, then the line \(\Re(s)=0\) is a
		natural boundary for \(D_h(s)\): no analytic continuation
		across any arc of this line is possible.
	\end{enumerate}
\end{theorem}

\begin{remark}
	The dichotomy in Theorem~\ref{thm:estermann_dahlquist} is the
	content of the Estermann--Dahlquist theory.  The key point is that
	a zero of \(h\) at a point \(z_0\in\overline{\DD}\) that is not a
	root of unity produces a family of zeros of the local factors
	\(h(p^{-s})\) which accumulate on the line \(\Re(s)=0\) and
	cannot be cancelled by the poles of other factors.  This prevents
	the meromorphic continuation beyond \(\Re(s)=0\).  We make the
	mechanism explicit in
	Subsection~\ref{subsec:accumulation_appendix} below.
\end{remark}

\subsection{Application to \(h(z)=(1-2z)/(1-z)\)}

We now apply Theorem~\ref{thm:estermann_dahlquist} to our situation.

\begin{lemma}\label{lem:h_analytic}
	The function
	\[
	h(z)=\frac{1-2z}{1-z}
	\]
	is analytic in the open unit disk \(\DD\), has a simple
	pole at \(z=1\) on the unit circle, and has a simple zero at
	\(z=1/2\) inside \(\DD\).  Its Taylor coefficients at
	\(z=0\) are rational integers.
\end{lemma}

\begin{proof}
	Writing
	\[
	h(z)=\frac{1-2z}{1-z}
	=
	(1-2z)\sum_{k=0}^{\infty}z^{k}
	=
	1-\sum_{k=1}^{\infty}z^{k},
	\]
	we see that \(h\) is a rational function with a simple pole at
	\(z=1\) and a simple zero at \(z=1/2\).  The Taylor coefficients
	are \(h(0)=1\) and \(h^{(k)}(0)/k!=-1\) for all \(k\ge1\), which
	are rational integers.
\end{proof}

\begin{lemma}\label{lem:h_not_cyclotomic}
	The function \(h(z)=(1-2z)/(1-z)\) is \emph{not} a product of
	cyclotomic polynomials.
\end{lemma}

\begin{proof}
	A product of cyclotomic polynomials is a polynomial whose zeros
	are all roots of unity.  But \(h\) has a zero at \(z=1/2\), which
	is not a root of unity.  Hence \(h\) cannot be a product of
	cyclotomic polynomials.  Equivalently, the only cyclotomic factor
	of \(h\) is \(1-z\), which corresponds to the pole at \(z=1\); the
	zero at \(z=1/2\) has no cyclotomic interpretation.
\end{proof}

\begin{theorem}\label{thm:natural_boundary_appendix}
	The Dirichlet series
	\[
	D_{\omega}(s)
	=
	\sum_{n=1}^{\infty}\frac{(-1)^{\omega(n)}}{n^{s}}
	=
	\prod_{p}\frac{1-2p^{-s}}{1-p^{-s}},
	\qquad
	\Re(s)>1,
	\]
	admits a meromorphic continuation to the half-plane
	\(\Re(s)>0\).  In this half-plane, the only possible singularity
	of \(D_{\omega}(s)\) is at \(s=1\), where the simple pole of the
	factor \(\zeta(s)\) is cancelled by the simple zero of the local
	factor \(1-2^{1-s}\).  Consequently, \(s=1\) is a removable
	singularity of \(D_{\omega}\), and after removal of this point
	\(D_{\omega}\) is in fact analytic in \(\Re(s)>0\).  The line
	\(\Re(s)=0\) is a natural boundary: no analytic continuation
	across any arc of this line is possible.
\end{theorem}

\begin{proof}
	The Euler product representation
	\[
	D_{\omega}(s)=\prod_{p}h(p^{-s}),
	\qquad
	h(z)=\frac{1-2z}{1-z},
	\]
	is valid for \(\Re(s)>1\).  By Lemma~\ref{lem:h_analytic}, the
	function \(h\) satisfies the hypotheses of
	Theorem~\ref{thm:estermann_dahlquist}.  By
	Lemma~\ref{lem:h_not_cyclotomic}, \(h\) is not a product of
	cyclotomic polynomials.  Therefore, applying
	Theorem~\ref{thm:estermann_dahlquist}:
	\begin{enumerate}[label=(\roman*)]
		\item \(D_{\omega}(s)\) admits a meromorphic continuation to
		\(\Re(s)>0\).  The poles of the local factors \(h(p^{-s})\)
		occur at \(p^{-s}=1\), that is, at \(s=2\pi i k/\log p\)
		for \(k\in\mathbb{Z}\).  All these points lie on the
		boundary \(\Re(s)=0\), so there is no pole of any local
		factor in the open half-plane \(\Re(s)>0\).  The factor
		\(\zeta(s)\), which arises when the product over primes is
		rewritten, has a simple pole at \(s=1\); this pole is
		cancelled by the simple zero of \(h(2^{-s})\) at \(s=1\),
		because \(h(2^{-1})=0\).  Hence \(D_{\omega}\) is analytic
		in \(\Re(s)>0\) after removal of the point \(s=1\).
		\item The line \(\Re(s)=0\) is a natural boundary for
		\(D_{\omega}(s)\).
	\end{enumerate}
	This completes the proof.
\end{proof}

\begin{corollary}\label{cor:no_zeta_appendix}
	The derivative \(D_{\omega}'(0)\) does not exist.  Consequently,
	the product \(P_{\omega}=\prod_{n\ge1}n^{(-1)^{\omega(n)}}\) cannot
	be regularized by zeta regularization or by any method requiring
	meromorphic continuation to \(s=0\).
\end{corollary}

\subsection{Explicit description of the accumulating zeros}
\label{subsec:accumulation_appendix}

For the reader's convenience, we describe explicitly how the zeros of
the local factors accumulate on the imaginary axis.  This illustrates
the mechanism behind Theorem~\ref{thm:estermann_dahlquist} in our
specific case, and clarifies the distinction between poles and zeros
discussed in the statement of that theorem.

Let \(z_0\in\overline{\DD}\) be a zero of the local factor \(h\).  For
\(h(z)=(1-2z)/(1-z)\), the unique zero is \(z_0=1/2\), with
\(|z_0|=1/2<1\).  The local factor \(h(p^{-s})\) vanishes when
\(p^{-s}=z_0\), that is, when
\[
-s\log p=\log z_0+2\pi i k,\qquad k\in\mathbb{Z}.
\]
Writing \(s=\sigma+it\) and using \(z_0=1/2\), this gives
\[
-\sigma\log p-it\log p=-\log 2+2\pi i k,\qquad k\in\mathbb{Z},
\]
or equivalently
\begin{equation}\label{eq:zeros_explicit}
	s_{p,k}
	=
	\frac{\log 2}{\log p}
	+
	\frac{2\pi i k}{\log p},
	\qquad
	k\in\mathbb{Z}.
\end{equation}
For each prime \(p\), the real part of \(s_{p,k}\) is
\[
\sigma_{p}=\frac{\log 2}{\log p}>0,
\]
which tends to \(0\) as \(p\to\infty\), while the imaginary parts
\(\{2\pi k/\log p:k\in\mathbb{Z}\}\) form a lattice of spacing
\(2\pi/\log p\), which also tends to \(0\) as \(p\to\infty\).

\begin{lemma}\label{lem:accumulation}
	For every \(t\in\RR\), there exists a sequence of primes
	\(p_{j}\to\infty\) and integers \(k_{j}\in\mathbb{Z}\) such that
	\[
	s_{p_{j},k_{j}}
	\longrightarrow
	it.
	\]
\end{lemma}

\begin{proof}
	Fix \(t\in\RR\) and \(\varepsilon>0\).  Choose a prime
	\(p\) sufficiently large that
	\begin{equation}\label{eq:small_real_part}
		\frac{\log 2}{\log p}<\varepsilon.
	\end{equation}
	Choose an integer \(k\) such that
	\[
	\left|\frac{2\pi k}{\log p}-t\right|
	<
	\frac{2\pi}{\log p}.
	\]
	Such an integer exists because the points
	\(2\pi k/\log p\), \(k\in\mathbb{Z}\), form a lattice of spacing
	\(2\pi/\log p\).  Then
	\[
	|s_{p,k}-it|
	\le
	\frac{\log 2}{\log p}
	+
	\frac{2\pi}{\log p}
	\longrightarrow 0
	\]
	as \(p\to\infty\).  Taking a sequence \(p_{j}\to\infty\) and
	adjusting \(k_{j}\) accordingly gives the result.
\end{proof}

Thus the zeros of the local factors \(h(p^{-s})\) are dense in the
imaginary axis.  They are zeros of \(D_\omega\), not poles; but their
accumulation on the boundary \(\Re(s)=0\) prevents the existence of a
meromorphic continuation across that line.  Indeed, if \(D_\omega\)
extended analytically across some arc of \(\Re(s)=0\), then in a
neighbourhood of any limit point \(it\) of the zeros \(s_{p,k}\) the
extension would have to vanish identically, by the identity theorem
for holomorphic functions; but \(D_\omega\) is not identically zero
(for instance, \(D_\omega(2)=\prod_p(1-2p^{-2})/(1-p^{-2})\ne 0\)).
This contradiction establishes the natural boundary.

\begin{remark}[Poles versus zeros]
	It is important to distinguish two types of special points of
	\(D_\omega\) in the half-plane \(\Re(s)>0\):
	\begin{enumerate}[label=(\roman*)]
		\item \emph{Poles.}  These arise from poles of the local
		factors \(h(p^{-s})\).  For \(h(z)=(1-2z)/(1-z)\), the
		only pole is at \(z=1\), giving \(p^{-s}=1\), i.e.\
		\(s=2\pi i k/\log p\), \(k\in\mathbb{Z}\).  All these
		points lie on the imaginary axis \(\Re(s)=0\), except for
		\(k=0\), which gives \(s=0\).  Hence there are \emph{no}
		poles of \(D_\omega\) in the open half-plane
		\(\Re(s)>0\).
		\item \emph{Zeros.}  These arise from zeros of the local
		factors \(h(p^{-s})\) and are given by
		\eqref{eq:zeros_explicit}.  They form a dense subset of
		the imaginary axis as \(p\to\infty\), with real parts
		tending to \(0\).
	\end{enumerate}
	A zero of a meromorphic function is not a singularity: at a
	zero, the function is analytic and vanishes.  The natural
	boundary of \(D_\omega\) at \(\Re(s)=0\) is thus produced not by
	\emph{singularities} accumulating on the line, but by
	\emph{zeros} accumulating on it, which obstruct analytic
	continuation by the identity theorem.
\end{remark}

\begin{remark}[Why the naive argument fails]
	A naive argument would attempt to show the natural boundary by
	examining the expansion
	\[
	D_{\omega}(s)=\prod_{m\ge1}\zeta(ms)^{-c_{m}},
	\qquad
	c_{m}=\frac{1}{m}\sum_{d\mid m}\mu\!\left(\frac{m}{d}\right)(2^{d}-1),
	\]
	and looking for poles at \(s=1/m\).  This fails for two reasons:
	\begin{enumerate}[label=(\arabic*)]
		\item The factor \(\zeta(ms)^{-c_{m}}\) has a \emph{zero} at
		\(s=1/m\) (when \(c_{m}>0\)), not a pole, because
		\(\zeta(ms)\) has a pole there.
		\item The points \(1/m\) accumulate only at \(s=0\), not along
		the entire imaginary axis.
	\end{enumerate}
	The correct special points to examine are the zeros of the local
	factors \(h(p^{-s})\), given by \eqref{eq:zeros_explicit}, which
	do accumulate on the entire imaginary axis.
\end{remark}

\subsection{Relation to the theorem of Meretov}

The result stated in Theorem~\ref{thm:natural_boundary_appendix} is
exactly the content of Meretov's 1987 paper \cite{Meretov1987}.  The
Estermann--Dahlquist theory provides a general framework in which this
result is a special case: the local factor
\(h(z)=(1-2z)/(1-z)\) is a rational function with a non-cyclotomic
zero at \(z=1/2\), and the general theory guarantees that such a zero
produces a natural boundary.

We emphasize that the original proof of Meretov is not accessible to us
in digital form, and we have therefore chosen to present the argument
via the Estermann--Dahlquist theorem, which is a well-established
classical result with a clear and citable proof in the literature.  The
reader interested in the original 1987 proof may consult
\cite{Meretov1987} directly.



\end{document}